\documentclass[submission%
]{dmtcs-episciences}

\usepackage[latin1]{inputenc}
\usepackage{subfigure}
\usepackage{rotating} 
\usepackage{amssymb}
\usepackage{amsmath}

\newtheorem{defi}{Definition}
\newtheorem{lem}{Lemma}
\newtheorem{cor}{Corollary}
\newtheorem{thm}{Theorem}

\newtheorem{prp}{Property}

\newtheorem{ex}{Example} 

\usepackage[round]{natbib}

\author{Jozef Kratica\affiliationmark{1}
  \and Vera Kova\v{c}evi\'c-Vuj\v{c}i\'c\affiliationmark{2}
  \and Mirjana \v{C}angalovi\'c\affiliationmark{2}}

\title{The strong metric dimension of some generalized Petersen graphs
\thanks{This research was partially supported by Serbian
Ministry of Education, Science and Technological Development under
the grants no. 174010 and 174033.}}

\affiliation{
  Mathematical Institute, Serbian Academy of Sciences and Arts\\
  Faculty of Organizational Sciences, University of Belgrade, Serbia}


\received{2015-07-03}
\revised{, , , }
\accepted{}

\begin{document}
\publicationdetails{VOL}{2015}{ISS}{NUM}

\maketitle

\begin{abstract}
In this paper the strong metric dimension of generalized Petersen
graphs $GP(n,2)$ is considered. The exact value is determined for
cases $n=4k$ and $n=4k+2$, while for $n=4k+1$ an upper bound of the
strong metric dimension is presented. 
\end{abstract}

\section{Introduction}

The strong metric dimension problem was introduced by \cite{seb04}. 
This problem is defined in the following way.
Given a simple connected undirected graph $G$ = ($V$,$E$), where
$V=\{1,2,\dots,n\}$, $|E|=m$ and $d(u,v)$ denotes the distance
between vertices $u$ and $v$, i.e. the length of a shortest $u-v$
path. A vertex $w$ strongly resolves two vertices $u$ and $v$ if $u$
belongs to a shortest $v-w$ path or $v$ belongs to a shortest $u-w$
path. A vertex set $S$ of $G$ is a {\em strong resolving set} of $G$
if every two distinct vertices of $G$ are strongly resolved by some
vertex of $S$. A {\em strong metric basis} of $G$ is a strong
resolving set of the minimum cardinality. The {\em strong metric
dimension} of $G$, denoted by $sdim(G)$, is defined as the
cardinality of its strong metric basis. Now, the strong metric
dimension problem is defined as the problem of finding the strong
metric dimension of graph $G$.

\begin{ex} Consider the Petersen graph G given on Figure 1. It is easy to see that
set $S=\{u_0,u_1,u_2,u_3,\\v_0,v_1,v_2,v_3\}$ is a strong resolving
set, i.e. each pair of vertices in $G$ is strongly resolved by a
vertex from $S$. Since any pair with at least one vertex in $S$ is
strongly resolved  by that vertex, the only interesting case is pair
$u_4,v_4$. This pair is strongly resolved e.g. by $u_0 \in S$ since
the shortest path $u_0,u_4,v_4$ contains $u_4$. In Section 2 it will
be demonstrated that $S$ is a strong resolving set with the minimum
cardinality and, therefore, $sdim(G)=8$.
\end{ex}

\begin{figure}
\begin{center}
\resizebox{5cm}{!}{\includegraphics{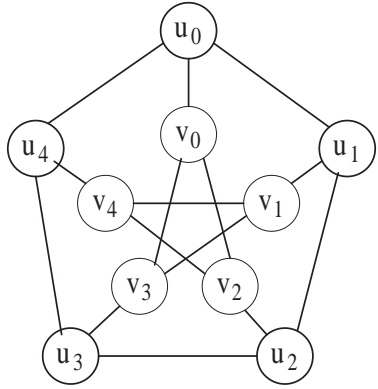}}
 \caption{Petersen graph G}
\end{center}
\end{figure}

The strong metric dimension has many interesting theoretical
properties. If $S$ is a strong resolving set of $G$, then the matrix
of distances from all vertices from $V$ to all vertices from $S$
uniquely determines graph $G$ \citep{seb04}.

The strong metric dimension problem is NP-hard in general case
\citep{oel07}. Nevertheless, for some classes of graphs the strong
metric dimension problem can be solved in polynomial time. For
example, in \citep{may11} an algorithm for finding the strong metric
dimension of distance hereditary graphs with $O(|V|\cdot |E|)$
complexity is presented.

In \citep{kra12a} an integer linear programming (ILP) formulation of
the strong metric dimension problem was proposed. Let variable $y_i$
determine whether vertex $i$ belongs to a strong resolving set $S$ or not,
i.e. $y_i  =\begin{cases}
1, & i \in S\\
0, & i \notin S
\end{cases}$. Now, the ILP model of the strong metric dimension problem
is given by (1)-(3).

\begin{equation}
\min \sum\limits_{i = 1}^n {y_i }
\end{equation}

subject to:

\begin{equation}
\sum\limits_{i = 1}^{n} {A_{(u,v),i} \cdot y_i}  \ge 1\quad \quad
\quad \quad \quad \quad 1 \le u < v \le n
\end{equation}

\begin{equation}
y_i  \in \{ 0,1\} \quad \quad \quad \quad \quad 1 \le i \le n
\end{equation}

where $A_{(u,v),i}  =\begin{cases}
1, & d(u,i) = d(u,v) + d(v,i)\\
1, & d(v,i) = d(v,u) + d(u,i)\\
0, & otherwise
\end{cases}$.\\

This ILP formulation will be used in Section 2 for finding the
strong metric dimension of generalized Petersen graphs in singular
cases, when dimensions are small. The following definition and two
properties from literature will also be used in Section 2.

\begin{defi}\citep{kra12a} A pair of vertices $u,v \in V$, $u \ne v$, is mutually maximally distant
if and only if
\begin{enumerate}[(i)]
    \item $d(w,v) \leq d(u,v)$ for each $w$ such that $\{w,u\} \in E$ and
    \item $d(u,w) \leq d(u,v)$ for each $w$ such that $\{v,w\} \in E$.
\end{enumerate}
\end{defi}

\begin{prp}\citep{kra12a} If $S \subset V$ is a strong resolving set of graph $G$, then,
for every two maximally distant vertices $u,v \in V$, it must be $u
\in S$ or $v \in S$.\end{prp}

Let $Diam(G)$ denote the diameter of graph $G$, i.e. the maximal
distance between two vertices in $G$.

\begin{prp}\citep{kra12a} If $S \subset V$ is a strong resolving set of graph $G$,
then, for every two vertices $u,v \in V$ such that $d(u,v) =
Diam(G)$, it must be $u \in S$ or $v \in S$.
\label{{prp:s1}}\end{prp}

A survey paper \citep{kra13} contains the results related to the
strong metric dimension of graphs up to mid 2013. Here we give a
short overview of the newest results which are not covered in
\citep{kra13}:
\begin{itemize}
\item In \citep{das14} an 2-approximation algorithm and a
($2-\epsilon$)-inapproximability result for the strong metric
dimension problem is given;
\item In \citep{yi13} a Nordhaus-Gaddum-type upper and lower bound for the
strong metric dimension of a graph and its complement is given. The
characterization of the cases when the bounds are attained is also
presented;
\item A comparison between the zero forcing number and the strong metric
dimension of graphs is presented by \cite{kan14};
\item Closed formulas for the strong metric dimension of several
families of the Cartesian product of graphs and the direct product
of graphs are given in \citep{rod13,kuz14b};
\item In \citep{kuz14b,kuz13} the authors study the strong metric dimension of
the rooted product of graphs and express it in terms of invariants
of the factor graphs;
\item Several relationships between the strong metric dimension of
the lexicographic product of graphs and the strong metric dimension
of its factor graphs are obtained in \citep{kuz14};
\item The strong metric dimension of the strong product of graphs is
analyzed in \citep{kuz14b,kuz15};
\item In \citep{sal14} the strong metric dimension of some convex polytopes
is obtained;
\item The strong metric dimension of tetrahedral diamond graphs
is found in \citep{raj};
\item The problem of finding the strong metric dimension of
circulant graphs and a lower bound for diametrically vertex uniform
graphs is solved by \cite{gri};
\item Some similar dimensions of graphs are introduced in literature:
the strong partition dimension \citep{yer13} and the fractional
strong metric dimension \citep{kan13}.
\end{itemize}

This paper considers the strong metric dimension of a special class
of graphs, so called generalized Petersen graphs. The generalized
Petersen graph $GP(n,k)$ ($n \geq 3$; $1 \leq k < n/2$) has $2n$
vertices and $3n$ edges, with vertex set $V$ = \{$u_i, v_i \;|\; 0
\leq i \leq n-1\}$ and edge set $E$ = \{$\{u_i,u_{i+1}\}$,
$\{u_i,v_i\}$, $\{v_i,v_{i+k}\}$ $|$ $0 \leq i \leq n-1$\}, where
vertex indices are taken modulo $n$. The Petersen graph from Figure
1 can be considered as $GP(5,2)$.

Generalized Petersen graphs were first studied by 
\cite{gpprvi}. Various properties of $GP(n,k)$ have been recently
theoretically investigated in the following areas: metric dimension \citep{naz14},
decycling number \citep{gao15}, component connectivity \citep{fer14} and
acyclic 3-coloring \citep{zhu14}.

In the case when $k=1$ it is easy to see that $GP(n,1) \equiv C_n
\Box P_2$, where $\Box$ is the Cartesian product of graphs, while
$C_n$ is the cycle on $n$ vertices and $P_2$ is the path with 2
vertices. Using the following result from \citep{rod13}: $sdim(C_n
\Box P_r)$ = $n$, for $r \geq 2$, it follows that $sdim(GP(n,1))=n$.

\section{The strong metric dimension of $GP(n,2)$}

In this section we consider the strong metric dimension of the
generalized Petersen graph $GP(n,2)$. The exact value is determined
for cases $n=4k$ and $n=4k+2$, while for $n=4k+1$ an upper bound of
the strong metric dimension is presented. In order to prove that the
sets defined in Lemma 1, Lemma 3 and Lemma 5 are strong resolving we
used shortest paths given in Tables \ref{l4k2a}-\ref{l4k1a}, which is organized as follows:
\begin{itemize}
\item First column named ''{\em case}'' contains
the case number;
\item Next two columns named ''{\em vertices}'' and ''{\em res. by}'' contain a pair of vertices and
a vertex which strongly resolves them;
\item Last two columns named ''{\em condition}'' and ''{\em shortest path}''
contain the condition under which the vertex in column five strongly
resolves the pair in column four and the corresponding shortest path,
respectively.
\end{itemize}

\begin{sidewaystable}
\small \caption{Shortest paths in Lemma \ref{4k2a}}
\label{l4k2a}
\begin{center}
\begin{tabular}{|l|l|c|c|c|}
\hline
  {\em Case } & {\em vertices} & {\em res. by} & {\em condition} & {\em shortest path} \\
 \hline
 1 & ($u_i,u_j$) , & $u_{i-1}$ & $j-i = 2$ & $u_{i-1}, u_i, u_{i+1}, u_{i+2}=u_j$\\
  & $i$,$j$ odd & $u_{i-1}$ & $4 \leq j-i \leq 2k$ & $u_{i-1}, u_i, v_i, v_{i+2}, ..., v_j, u_j$\\
\hline
  2 & ($v_i,v_j$), & $u_j$ and $u_{j+1}$ & $j-i$ even & $v_i, v_{i+2}, ..., v_j, u_j, u_{j+1}$ \\
 &  $i,j \geq 2k+1$ & $v_{2k}$ & $j-i$ odd, $i$ even & $v_{2k}, v_{2k+2}, ..., v_i, u_i, u_{i+1},v_{i+1}, v_{i+3}, ..., v_j$ \\
  &  & $v_{2k-1}$ & $j-i$ odd, $i$ odd & $v_{2k-1}, v_{2k+1}, ..., v_i, u_i, u_{i+1},v_{i+1}, v_{i+3}, ..., v_j$ \\
\hline
  3 & ($u_i,v_j$),  & $u_{i+1}$ & $i \geq j$, $j$ odd & $v_j, v_{j+2}, ..., v_j, v_{j+2}, ..., v_i, u_i,u_{i+1}$\\
  & $i$ odd, $j \geq 2k+1$& $v_{2k}$ & $i > j$, $j$ even & $v_{2k}, v_{2k+2}, ..., v_j, u_j, u_{j+1}, v_{j+1}, v_{j+3}, ..., v_i, u_i$\\
  & & $u_{i-1}$ & $j$ odd,$0 < j-i \leq 2k$ & $u_{i-1}, u_i, v_i, v_{i+2},...,v_j$\\
  & & $u_{i+1}$ & $j$ odd,$j-i \geq 2k+2$ & $v_j, v_{j+2}, ...., v_i, u_i, u_{i+1}$\\
 & & $v_{2k-1}$ & $j-i = 1$ & $v_{2k-1}, v_{2k+1}, ..., v_i, u_i, u_{i+1},v_{i+1}=v_j$\\
 & & $u_{j+2}=u_{i+5}$ & $j-i = 3$ & $u_i, u_{i+1}, v_{i+1}, v_{i+3}=v_j, v_{j+2}, u_{j+2}$\\
  & & $u_j$ & $j$ even, $5 \leq j-i \leq 2k+1$ & $u_i, u_{i+1}, v_{i+1}, v_{i+3}, ..., v_j, u_j$\\
  & & $u_j$ & $j$ even, $2k+3 \leq j-i < 4k-1$ & $u_j, v_j, v_{j+2}, ..., v_{i-1}, u_{i-1}, u_i$\\
  & & $v_1$ & $j$ even, $j-i = 4k-1$ & $v_j=v_{4k}, v_{j+2}=v_0, u_0, u_1=u_i, v_1$\\
\hline
 \end{tabular}
\end{center}
\end{sidewaystable}

\begin{lem} \label{4k2a}Set $S = \{ u_{2i}, v_i | i=0,1,\dots,2k\}$ is a strong resolving set of $GP(4k+2,2)$ for $k \geq 3$.\end{lem}
\begin{proof}Let us consider pairs of vertices such that neither
vertex is in $S$. There are three possible cases:
\begin{itemize}
\item[Case 1.] $(u_i,u_j)$, $i,j$ are odd. Without loss of generality we
may assume that $i < j$. If $j - i \leq 2k$, according to Table \ref{l4k2a}, vertices $u_i$ and
$u_j$ are strongly resolved by vertex $u_{i-1}$. The shortest
paths corresponding to the subcases $j-i=2$ and $4 \leq j-i \leq 2k$
are also given in Table \ref{l4k2a}. As $i-1$ is even, then $u_{i-1} \in S$.
If $j - i \geq 2k+2$, then pair $(u_i,u_j)$ can be represented as
pair $(u_{i'},u_{j'})$, where $i'=j$ and $j'=i+4k+2$. Since
$j' - i' \leq 2k$, the situation is reduced to the previous one.

\item[Case 2.] $(v_i,v_j)$, $i,j \geq 2k+1$. Without loss of generality we
may assume that $i < j$. If $j-i$ is even, $v_i$ and $v_j$ are
strongly resolved by both $u_j$ and $u_{j+1}$. When $j$ is even $u_j
\in S$, while $u_{j+1} \in S$ for odd $j$. If $j-i$ is odd, then
$v_i$ and $v_j$ are strongly resolved by $v_{2k} \in S$ if $i$ is
even, and $v_{2k-1} \in S$ if $i$ is odd, and hence $u_i$ and $u_j$
are strongly resolved by $S$;

\item[Case 3.] $(u_i,v_j)$, $i$ odd, $j \geq 2k+1$. There are
nine subcases, characterized by conditions presented in Table
\ref{l4k2a}. Vertices which strongly resolve pair ($u_i,v_j$) listed
in Table \ref{l4k2a} belong to set S. Indeed, vertices $v_1$,
$v_{2k-1}$ and $v_{2k}$ belong to $S$ by definition, while
$u_{i-1}$, $u_{i+1}$ and $u_{i+5}$ belong to $S$ since $i$ is odd.
Finally, $u_j \in S$ since $j$ is even by condition.
\end{itemize}
\end{proof}

\begin{lem} \label{4k2b}If $S$ is a strong resolving set of $GP(4k+2,2)$, then $|S| \geq 4k+2$, for any $k \geq 2$.\end{lem}
\begin{proof} Since $d(u_i,u_{i+2k+1})=d(v_i,v_{i+2k+1})=k+3 =
Diam(GP(4k+2,2)), i=0,1,...,2k$, from Property 2, at least $2k+1$
u-vertices and $2k+1$ v-vertices belong to $S$. Therefore, $|S| \geq
4k+2$.
\end{proof}

The strong metric dimension of $GP(4k+2,2)$ is given in Theorem \ref{4k2}.

\begin{thm} \label{4k2} For all $k$ it holds that $sdim(GP(4k+2,2))=4k+2$.\end{thm}
\begin{proof}
It follows directly from Lemmas \ref{4k2a} and \ref{4k2b} that $sdim(GP(4k+2,2)) = 4k+2$ for $k \geq 2$.
For $k=1$, using CPLEX solver on ILP formulation (1)-(3), we have proved that set $S$ from Lemma \ref{4k2a} is also a
strong metric basis of $GP(6,2)$, i.e. $sdim(GP(6,2))=6$.
\end{proof}

\begin{sidewaystable}
\small \caption{Shortest paths in Lemma \ref{4ka}}
\label{l4ka}
\begin{center}
\begin{tabular}{|l|l|c|c|c|}
\hline
  {\em Case } & {\em vertices} & {\em res. by} & {\em condition} & {\em shortest path} \\
 \hline
 1 & ($u_i,u_j$), & $u_{j+1}=u_{i+3}$ & $j-i = 2$ & $u_i, u_{i+1}, u_{i+2}=u_j, u_{j+1}$\\
  & $i,j \geq 2k$,$i,j$ even &$u_{j+1}$ & $j-i \geq 4$ & $u_i, v_i, v_{i+2}, ..., v_j, u_j, u_{j+1}$\\
 \hline
 2 & ($v_i,v_j$), & $u_{j+1}$ & $j-i \leq 2k-2$  & $v_i, v_{i+2}, ..., v_j, u_j, u_{j+1}$\\
 & $i,j$ even & $u_i$ and $u_{i+2k}$ & $j-i = 2k$  & $u_i, v_i, v_{i+2}, ..., v_{i+2k}=v_j, u_{i+2k}=u_j$\\
 \hline
 3 & ($u_i,v_j$), & $u_{i-1}$& $0 \leq j-i \leq 2k-2$ & $ u_{i-1}, u_i, v_i, v_{i+2}, ..., v_j$\\
  & $i,j$ even,  & $u_{i-2k} = u_j$& $j-i = -2k$ & $ u_i, v_i, v_{i-2}, ..., v_{i-2k}=v_j, u_{i-2k}=u_j$\\
  & $i \geq 2k$ & $u_{j-1}$ & $-2k < j - i < -2$ & $ u_{j-1}, u_j, v_j, v_{j+2}, ..., v_i, u_i$\\
  & & $u_{i+1}$ & $j - i = -2$ & $v_j, v_{j+2}=v_i,u_i,u_{i+1}$\\
\hline
 \end{tabular}
\end{center}
\end{sidewaystable}

The strong metric dimension of $GP(4k,2)$ will be determined using
Lemmas \ref{4ka} and \ref{4kb}.

\begin{lem}\label{4ka} Set $S= \{u_i|i=0,1,...,2k-1\} \cup \{u_{2k+2i+1}| i=0,1,...k-1\} \cup \{v_{2i+1}| i=0,1,...2k-1\}$
is a strong resolving set of $GP(4k,2)$ for $k \geq 3$.\end{lem}
\begin{proof}
Since $S$ contains $u_i$, $i=1,...,2k-1$ and all $u_i$ and $v_i$ for odd $i$,
we need to consider only three possible cases:
\begin{itemize}
\item[Case 1.] $(u_i,u_j)$, $i,j \geq 2k$ are even. Without loss of generality we
may assume that $i < j$. Vertices $u_i$ and $u_j$ are strongly
resolved by vertex $u_{j+1}$ (see Table \ref{l4ka}). The shortest paths corresponding
to subcases $j-i=2$ and $j-i \geq 4$ are given in Table \ref{l4ka}. As $j+1$ is odd,
then $u_{j+1} \in S$;

\item[Case 2.] $(v_i,v_j)$, $i,j$ are even. Without loss of generality we
may assume that $i < j$. If $j-i \leq 2k-2$, vertices $v_i$ and
$v_j$ are strongly resolved by $u_{j+1} \in S$. If $j - i \geq 2k+2$,
then pair $(v_i,v_j)$ can be represented as
pair $(v_{i'},v_{j'})$, where $i'=j$ and $j'=i+4k$.
Since $j' - i' \leq 2k-2$, the situation is reduced to the previous one.
If $j-i = 2k$, vertices $v_i$ and
$v_j=v_{i+2k}$ are strongly resolved by both $u_i$ and
$u_j=u_{i+2k}$. Since $S$ contains $u_0, ..., u_{2k-1}$ it follows
that $u_i$ or $u_{i+2k}$ belongs to $S$ and hence $v_i$ and $v_j$
are strongly resolved by $S$;

\item[Case 3.] $(u_i,v_j)$, $i,j$ even, $i \geq 2k$. Let us assume first that $i \leq j$.
If $j-i \leq 2k-2$, vertices $u_i$ and $v_j$ are strongly resolved
by vertex $u_{i-1}$. As $i-1$ is odd it follows that $u_{i-1} \in S$
and $u_i$ and $v_j$ are strongly resolved by $S$. Assume now that $i
> j$. If $j-i = -2k$ then vertices $u_i$ and $v_j$=$v_{i-2k}$ are
strongly resolved by vertex $u_j$=$u_{i-2k}$. Since $S$ contains
$u_0, ..., u_{2k-1}$ and $i \geq 2k$, it follows that $u_{i-2k}$
belongs to $S$ and hence $u_i$ and $v_j$ are strongly resolved by
$S$.  If $-2k < j-i < -2$ then pair ($u_i,v_j$) is strongly resolved
by $u_{j-1}$. Since $j-1$ is odd, then $u_{j-1} \in S$. Finally, if
$j-i=-2$ then pair of vertices $(u_i,v_j)$ is strongly resolved by
$u_{i+1} \in S$. These four subcases cover all possible values for
$i$ and $j$ having in mind that vertex indices are taken modulo $n$.
\end{itemize}
\end{proof}

\begin{lem}\label{4kb} If $k \geq 10$ and $S$ is a strong resolving set of $GP(4k,2)$, then $|S| \geq 5k$.\end{lem}
\begin{proof} Let us note that $d(v_i,v_{i+2k-1})=k+2 =
Diam(GP(4k,2)), i=0,1,...,4k-1$. If we suppose that $S$ contains
less than $2k$ $v$-vertices, since we have $4k$ pairs
$(v_i,v_{i+2k-1}), i=0,...,4k-1$, and each vertex appears exactly
twice, there exists some pair $(v_i,v_{i+2k-1}), v_i \notin
S,v_{i+2k-1} \notin S$. This is in contradiction with Property 2.
Therefore, $S$ contains at least $2k$ $v$-vertices.

Considering $u$-vertices, we have two cases:
\begin{itemize}
\item[Case 1.] If there exist $u_{2i} \notin S$ and $u_{2j+1} \notin S$,
then, as pairs $\{u_{2i},u_{2i+2l-1}\}, l=3,4,...,2k-2$ and $\{u_{2j+1},u_{2j+2l}\}, l=3,4,...,2k-2$
are mutually maximally distant, at most 8 additional vertices are not in $S$:
$u_{2i-3}, u_{2i-1}$, $u_{2i+1},u_{2i+3},u_{2j-2},u_{2j},u_{2j+2},u_{2j+4}$.
Consequently, at most 10 vertices, where $10 \leq k$, are not in $S$;

\item[Case 2.] Indices of $u$-vertices which are not in $S$ are all either
even or odd. Without loss of generality we may assume that all these indices
are even. Since $d(u_{2i},u_{2i+2k})=k+2 =
Diam(GP(4k,2)), i=0,1,...,k-1$, according to Property 2, we have $k$ pairs $(u_{2i},u_{2i+2k}), i=0,...,k-1$,
with at most one vertex not in $S$. Therefore, at most $k$ $u$-vertices are not in $S$.
\end{itemize}

In both cases we have proved that at most $k$ $u$-vertices are not in $S$, so
at least $3k$ $u$-vertices are in $S$. Since we have already proved that
at least $2k$ $v$-vertices should be in $S$, it follows that $|S| \geq 5k$.
\end{proof}

The strong metric dimension of $GP(4k,2)$ is given in Theorem \ref{4k}.

\begin{thm} \label{4k} For all $k \geq 5$ it holds $sdim(GP(4k,2))=5k$.\end{thm}
\begin{proof}
Lemma \ref{4ka} and Lemma \ref{4kb} imply that $S= \{u_i|i=0,1,...,2k-1\} \cup
\{u_{2k+2i+1}| i=0,1,...k-1\} \cup \{v_{2i+1}| i=0,1,...2k-1\}$ is a
strong metric basis of $GP(4k,2)$ for $k \geq 10$. Using CPLEX
solver on ILP formulation (1)-(3), we have proved that set $S$ from Lemma \ref{4ka} is a
strong metric basis of $GP(4k,2)$ for $k \in \{5,6,7,8,9\}$.
\end{proof}

\begin{sidewaystable}
\small \caption{Shortest paths in Lemma \ref{4k1a}}
\label{l4k1a}
\begin{center}
\begin{tabular}{|l|l|c|c|c|}
\hline
  {\em Case } & {\em vertices} & {\em res. by} & {\em condition} & {\em shortest path} \\
 \hline
 1 & ($u_{2i},u_{2j}$), & $u_{2j+1}$ & $j-i \geq 2$ & $ u_{2i}, v_{2i}, v_{2i+2}, ..., v_{2j}, u_{2j}, u_{2j+1}$ \\
  & $0 \leq i,j \leq k-1$ &$u_{2j+1}$ & $j-i = 1$ & $u_{2i}, u_{2i+1}, u_{2i+2}=u_{2j}, u_{2i+3}=u_{2j+1}$\\
\hline
  2 & ($v_i,v_j$), & $v_1$ & $i,j$ even & $ v_i, v_{i+2}, ..., v_j, v_{j+2}, ..., v_1$\\
   & $2k+4 \leq i,j \leq 4k$ & $v_0$ & $i,j$ odd & $ v_i, v_{i+2}, ..., v_j, v_{j+2}, ..., v_0$\\
   & & $v_0$ & $i$ even, $j$ odd & $ v_i, u_i, u_{i+1}, v_{i+1}, v_{i+3}, ..., v_j, v_{j+2}, ..., v_0$\\
  & & $v_1$ & $i$ odd, $j$ even & $ v_i, u_i, u_{i+1}, v_{i+1}, v_{i+3}, ..., v_j, v_{j+2}, ..., v_1$\\
\hline
 3 & ($u_{2i},v_j$), & $u_j$& $j$ even, $j-2i \leq 2k$ & $ u_{2i}, v_{2i}, v_{2i+2}, ..., v_j, u_j$ \\
  & $0 \leq i \leq k-1$, & $u_j$& $j$ even, $2k+2 \leq j-2i < 4k-2$ & $ u_j, v_j, v_{j+2}, ..., v_{2i-1}, u_{2i-1}, u_{2i}$ \\
  & $2k+4 \leq j \leq 4k$& $u_j$& $j$ odd, $j-2i \leq 2k-1$ & $ u_{2i}, u_{2i+1}, v_{2i+1}, v_{2i+3}, ..., v_j, u_j$ \\
  & & $u_j$& $j$ odd, $2k+1 \leq j-2i \leq 4k-3$ & $ u_j, v_j, v_{j+2}, ..., v_{2i}, u_{2i}$ \\
  & & $v_{2i}$& $j-2i = 4k-2$ & $ v_j, v_{j+2}, u_{j+2}, u_{j+3}, v_{j+3}=v_{2i}$ \\
  & & $u_1$& $j-2i = 4k-1$ & $ v_j=v_{4k-1}, v_{j+2}=v_0=v_{2i}, u_0=u_{2i}, u_1$ \\
  & & $v_0$& $j-2i = 4k$ & $ v_j=v_{4k}, u_j=u_{4k}, u_0=u_{2i}, v_0$ \\
\hline
 \end{tabular}
\end{center}
\end{sidewaystable}

In the case when $n=4k+1$, an upper bound of the strong metric
dimension of $GP(4k+1,2)$ will be determined as a corollary of the
following lemma.

\begin{lem}\label{4k1a} Set $S= \{u_{2i+1}|i=0,1,...,k-1\} \cup \{u_{2k+i}| i=0,1,...,2k\} \cup \{v_i| i=0,1,...2k+3\}$
is a strong resolving set of $GP(4k+1,2)$ for $k \geq 3$.\end{lem}
\begin{proof}
As in Lemma \ref{4k2a}, there are three possible cases:
\begin{itemize}
\item[Case 1.] $(u_{2i},u_{2j})$, $0 \leq i,j \leq k-1$. Without loss of generality we
may assume that $i < j$. Vertices $u_{2i}$ and $u_{2j}$ are strongly
resolved by vertex $u_{2j+1} \in S$. The shortest paths
corresponding to subcases $j-i \geq 2$ and $j-i=1$ are given in
Table \ref{l4k1a};

\item[Case 2.] $(v_i,v_j)$, $2k+4 \leq i,j \leq 4k$ .
Without loss of generality we may assume that $i < j$. If $j$ is
even, vertices $v_i$ and $v_j$ are strongly resolved by $v_1 \in S$,
while if $j$ is odd, they are strongly resolved by $v_0 \in S$. The
details about shortest paths can be seen in Table \ref{l4k1a};

\item[Case 3.] $(u_{2i},v_j)$, $0 \leq i \leq k-1$, $2k+4 \leq j \leq
4k$. In four initial subcases, when $j-2i < 4k-2$, vertices $u_{2i}$
and $v_j$ are strongly resolved by $u_j$. Since set $S$ contains
u-vertices $u_{2k}, u_{2k+1}, ..., u_{4k}$ and $2k+4 \leq j \leq 4k$
it follows that $u_j \in S$. The remaining subcases correspond to
situation when $j-2i \geq 4k-2$. The pair $(u_{2i},v_j)$ is strongly
resolved by $v_{2i}$ for $j-2i=4k-2$, by $u_1$ for $j-2i=4k-1$ and
by $v_0$ for $j-2i=4k$. By definition of $S$ it is obvious that
$u_1,v_0 \in S$, while $v_{2i} \in S$ since $2i \leq 2k-2$ and all
$v$-vertices with indices less or equal to $2k+3$ are in $S$. The
shortest paths corresponding to these seven subcases can be seen in
Table \ref{l4k1a}.
\end{itemize}
\end{proof}

\begin{cor} \label{4k1} If $k \geq 3$ then $sdim(GP(4k+1,2)) \leq 5k+5$.\end{cor}

The strong metric bases given in Lemma \ref{4k2a} and Lemma \ref{4ka} and the strong resolving
set given in Lemma \ref{4k1a} hold for $n \geq 20$. The strong metric bases for $n \leq 19$
have been obtained by CPLEX solver on ILP formulation (1)-(3). Computational results  show that
for $n \in \{6,10,14,18\}$ strong resolving set $S$ from Lemma \ref{4k2a} is a strong metric basis.
The strong metric basis for the remaining cases for $n \leq 19$ are
given in Table \ref{pos}.\\

\begin{table}
\small
\caption{Other strong metric bases of $GP(n,2)$ for $n \leq 19$}
\label{pos}
\begin{center}
\begin{tabular}{|c|c|c|}
\hline
 $n$ & $sdim(GP(n,2))$ & $S$ \\
 \hline
 5 & 8 & $\{u_0, u_1,u_2,u_3, v_0, v_1, v_2, v_3\}$ \\   
 \hline
 7& 9 &$\{u_0, u_1,u_2,u_3, v_0, v_1,v_2,v_3,v_4,v_6\}$ \\   
 \hline
 8& 8 &$\{u_4,u_5,u_6,u_7,v_1,v_3,v_5,v_7\}$ \\   
 \hline
 9 & 13 &$\{u_2,u_4,u_5,u_6,u_7,u_8, v_0,v_2,v_3,v_4,v_5,v_6,v_7\}$ \\  
 \hline
 11 & 12 &$\{u_0, u_1,u_2,u_3,u_4,u_5,v_0,v_1,v_2,v_5,v_6,v_7\}$ \\ 
 \hline
 12 & 13 & $\{u_0,u_1,u_2,u_3,u_4,u_5,u_6,v_0,v_2,v_4,v_6,v_8,v_{10}\}$ \\ 
 \hline
 13 & 17 &$\{u_0,u_1,u_2,u_3,u_4,u_5,u_6,u_7\}$ \\ 
 & & $\cup \{v_1,v_3,v_5,v_6,v_7,v_8,v_9,v_{10},v_{12}\}$ \\
 \hline
 15 & 20 &$\{u_0,u_1,u_2,u_3,u_4,u_5,u_6,u_7,u_8,u_9\}$ \\ 
 & & $\cup \{v_0,v_1,v_2,v_3,v_4,v_5,v_6,v_7,v_8,v_9\}$ \\
 \hline
 16 & 19 &$\{u_0,u_1,u_2,u_3,u_4,u_5,u_6,u_7,u_8,u_9,u_{10}\}$ \\ 
 & & $\cup \{v_0,v_2,v_4,v_6,v_8,v_{10},v_{12},v_{14}\}$ \\
 \hline
 17 & 24 &$\{u_0,u_1,u_2,u_3,u_4,u_5,u_6,u_7,u_8,u_9,u_{10},u_{11}\}$ \\ 
 & & $\cup \{v_0,v_1,v_2,v_3,v_4,v_5,v_6,v_7,v_8,v_9,v_{10},v_{11}\}$ \\
 \hline
 19 & 24 & $\{u_0,u_4,u_6,u_7,u_8,u_9,u_{10},u_{11},u_{12},u_{13},u_{14},u_{15},u_{16},u_{17}\}$ \\
 & & $\cup \{v_0,v_1,v_5,v_6,v_9,v_{10},v_{11},v_{14},v_{15},v_{16}\}$ \\
 \hline
 \end{tabular}
\end{center}
\end{table}

{\bf Open problem 1}. We beleive that the strong metric dimension of
$GP(4k+1,2)=5k+5$, $k \geq 5$, but we were not able to give a
rigorous proof.

{\bf Open problem 2}. For the remaining case $n=4k+3, k \geq 5$,
experimental results indicate the following hypothesis:
$sdim(GP(4k+3,2))= 5k+6$ for $k=5l-2$ and $sdim(GP(4k+3,2))= 5k+4$
for $k \neq 5l-2$, where $l \in N $.

\section{Conclusions}

In this paper we have studied the strong metric dimension of
generalized Petersen graphs $GP(n,2)$. We have found closed formulas
of the strong metric dimensions in cases $n=4k$ and $n=4k+2$,
and a tight upper bound of the strong metric dimension for $n=4k+1$.

A future work could be directed towards obtaining the strong metric
dimension of some other challenging classes of graphs.

\nocite{*}
\bibliographystyle{abbrvnat}
\bibliography{paper}

\end{document}